\documentclass[12pt]{amsart} 
\usepackage{amssymb,latexsym,amsfonts,verbatim,amscd,ifthen, graphicx} 
\usepackage{color,mathtools,tikz}

\newcommand{\C}{{\mathbb C}}

\numberwithin{equation}{section} 
 
\theoremstyle{plain} 
 
\newtheorem{theorem}[equation]{Theorem} 
\newtheorem{lemma}[equation]{Lemma} 
\newtheorem{proposition}[equation]{Proposition} 
\newtheorem{corollary}[equation]{Corollary} 
 
\newtheorem*{definition*}{Definition} 
\newtheorem*{theoremA}{Theorem A}

\theoremstyle{definition}

\newtheorem{definition}[equation]{Definition}

\usepackage{amsmath}
\usepackage[mathscr]{euscript}
\usepackage{comment}
\usepackage{amssymb}
\usepackage{amsthm}
\usepackage{faktor}
\usepackage{tikz}

\theoremstyle{definition}

\theoremstyle{remark}
 
 \numberwithin{equation}{section}
\begin{document}

\title[Spectral hypersurfaces and Hadamard matrices]{\bf Spectral hypersurfaces for operator pairs and Hadamard matrices of F type}

\author{T. Peebles}
\address{Department of Mathematics and Statistics \\ University at Albany\\Albany, NY 1222}
\email{tpeebles@albany.edu}
%\email{mstessin@albany.edu}
\author{M. Stessin}
\address{Department of Mathematics and Statistics \\ University at Albany \\ Albany, NY 12222}
\email{mstessin@albany.edu}
\begin{abstract}
We prove that if for a pair of $n\times n$ matrices $A$ and $B$ the projective joint spectrum  
of $A,B$, $AB$ and the identity is given by $$\sigma(A,B,AB,I)=\{[x,y,z,t]\in \C{\mathbb P}^3:x^n+y^n+(-1)^{n-1}z^n-t^n=0 \},$$
then this pair is unitary equivalent to a one associated with a complex Hadamard matrix of order $n$. If $n=3,4$, or 5, where there is a complete description of Hadamard matrices, we list those that generate a pair with the above mentioned spectrum. If 
\begin{align*}&\sigma(A,B,AB,BA,I) =\{[x,y,z_1,z_2,t]\in \C\mathbb{ P}^4:  \\ &x^n+y^n+(-1)^{n-1}(e^{2\pi i/n}z_1+z_2)^n-t^n=0\},\end{align*}
this Hadamard matrix is exactly the Fourier matrix $F_n$. If for an operator pair $A,B$ acting on a Hilbert space, such hypersurfaces appear as components of the projective joint spectrum of the corresponding tuples, then under some mild conditions  the pair has a common invariant subspace of dimension $n$, and the restriction of $A,B$ to this subspace is generated by a Hadamard matrix of F type. 
%It is well-known that, in general, an appearance of an algebraic hypersurface in the projective joint spectrum of an operator tuple does not imply the existence of a finite-dimensional common invariant subspace. We prove that if for a pair of operators $A,B$ the projective joint spectrum of $A,B$, and $AB$  contains the surface $\{[x,y,z,t]\in \C{\mathbb P}^3:x^n+y^n+(-1)^{n-1}z^n-t^n=0 \}$, then under some mild conditions this implies the existence of a subspace of dimension $n$ invariant for both $A$ and $B$. It is shown that the appearance of this surface has a relation to complex Hadamard matrices. We give a sufficient condition for a Hadamard matrix of F type to generate such pair $A,B$. For dimensions $n=3,4,5$, where there is a complete description of complex Hadamard matrices, this condition proved to be necessary as well. Finally we prove that a pair $A,B$ such that the projective joint spectrum of $A,B,AB$, and $BA$ contains
%$\{[x,y,z,t]\in \C{\mathbb P}^3:x^n+y^n+(-1)^{n-1}z^n-t^n=0 \}$$, is generated by the Fourier matrix $F_n$.  	
\end{abstract}
\maketitle

%\section{Main theorem}

\vspace{.2cm}

\section{Introduction and Statement of main results}

A matrix $h\in M_n(\C)$ is called a \textit{complex Hadamard matrix}, if  each entry of $h$ is a complex number of modulus one, and rows of the matrix are mutually orthogonal. Of course, this is equivalent to $\frac{1}{\sqrt{n}}h$ being unitary and all entries having the same absolute value. According to \cite{BN} these matrices were originally introduced by Sylvester \cite{S} as real matrices with entries $\pm1$ and orthogonal rows. Hadamard matrices with entries being roots of unity, denoted by $H(q,N)$ type, where $q$ is the order of the root of unity, and $N$ is the size of the matrix, were introduced by Butson \cite{Bu}. Matrices with arbitrary complex unimodular entries appeared in \cite{Po}. 

Interest to Hadamard matrices is caused by the fact that they appear in a number of mathematical objects. Among them 
are:  maximal abelian $*$-sublagebras of the algebra of complex $n\times n$ matrices \cite{Po},  statistical mechanical model, knot invariants, and planar algebras, \cite{J}, to name a few.  They are also associated with quantum permutation groups (cf  \cite{B1}, \cite{B}, \cite{BN}). A list of known families of Hadamard matrices can be found in \cite{TZ}. 

As we will see below, complex Hadamard matrices naturally appear in multivariable spectral theory in connection with projective joint spectra of operator pairs. This connection allows us to establish a spectral characterization of  Fourier type Hadamard matrices associated with a certain subgroup of permutation matrices. It also turned out that under some mild conditions, if the projective joint spectrum of a tuple generated by an operator pair acting on a Hilbert space contains as a component the hypersurface associated with the Fourier matrix of order $n$, then  this pair has a common invariant subspace of the same dimension $n$. The investigation of spectral properties of F type Hadamard matrices is the goal of the present paper.

Two $n\times n$  Hadamard matrices are called similar, $h_1\approx h_2$ , if

\begin{equation}\label{hadamard} h_2=\Lambda_1 P_1 h_1  P_2 \Lambda_2, \end{equation}
where  $\Lambda_1$ and $\Lambda_2$ are diagonal matrices with entries on the main diagonal having absolute value one, and $P_1$ and $P_2$ are permutation matrices of order $n$. Obviously, if $h_1$ is a Hadamard matrix, and $h_2$ is given by \eqref{hadamard}, then $h_2$ is Hadamard as well.

To the best of our knowledge a complete description of complex Hadamard matrices is known only for orders 2,3,4, and 5. The only family of matrices that appears among Hadamard matrices  in all dimensions is the family of matrices similar to $H(N,N)=F_N$ type. They are called Fourier, or F-type, matrices because of the relation of $F_N$ to the Fourier transform on ${\mathbb Z}_n$.  The matrix $F_N$ is  given by 
\begin{equation} \label{Fourier}
F_N=[f_{jk}]_{j,k=1}^N, \ f_{jk}=e^{2\pi (k-1)(j-1)i/N}.
\end{equation}
\vspace{.2cm}
There are also other known families in some higher dimensions, see \cite{Pe}, \cite{Ta}, and also \cite{TZ}. Self-adjoint $6\times 6$ complex Hadamard matrices were described in \cite{BeN}.
 
For $n=2,3,5$ every complex Hadamard matrix is similar to the Fourier matrix of the corresponding order. For $n=4$ every complex Hadamard matrix is similar to
\begin{equation}\label{n=4}
\left [ \begin{array}{cccc} 1 & 1 & 1 & 1 \\ 1 &t & -1 &-t \\ 1 &-1 & 1 & -1 \\ 1 &-t & -1 & t \end{array} \right ], 
\end{equation}
where $t$ is an arbitrary complex number of absolute value one. When $t=\ i$, the corresponding matrix is $F_4$, when $t=-i$, the matrix is similar to $F_4$. These results for $n=2,3,4$ are simple, and for $n=5$ it was proved by Haagerup in \cite{H}.

As stated earlier, we investigate an interplay between multivariable spectral theory and Hadamard matrices. While the classical spectral theory of normal operators acting on Hilbert spaces is a powerful tool for investigation in many areas of analysis, it took some time to find a good and sufficiently general definition of the spectrum of an operator tuple. 
%It also gives a solution of the following inverse problem: given the spectrum (with multiplicities) of a normal operator, reconstruct the operator up to a unitary equivalence. Both present a powerful tool of  investigation in various areas of analysis.
%A direct corollary to the spectral theorem is the fact that two normal operators with the same spectrum, counting multiplicities (which might be infinite), are unitary equivalent.
%A search for possible analogs of the for tuples of operators 
It started in the 1960s and in 1970 Taylor \cite{T} introduced a definition for commuting tuples, (see also \cite{EP}), and several other definitions followed (cf \cite{MP, P1,P2}). For an arbitrary, not necessarily commuting, tuple the notion of the \textit{projective joint spectrum} was introduce in \cite{Y}. This definition is a natural generalization of the classical definition of the spectrum. It is as follows:
\begin{definition} 
Let $A_1,...,A_n$ be operators acting on a Hilbert space $H$. The \textit{projective joint spectrum} of the tuple $A_1,...,A_n$, \ $\sigma(A_1,...,A_n)$, is defined as
$$\sigma(A_1,...,A_n)=\{ [x_1,...,x_n]\in \C{\mathbb P}^{n-1}: x_1A_1+...+x_nA_n \ \mbox{is not invertible} \}.  $$	
\end{definition}
\vspace{.2cm}
If the dimension of $H$ is finite, the above definition turns into
$$\sigma(A_1,...,A_n)=\{[x_1,...,x_n]\in \C{\mathbb P}^{n-1}: det (x_1A_1+...+x_nA_n)=0 \},$$
and the projective joint spectrum turns into the determinantal variety of a matrix pencil, which has been under scrutiny for more than a hundred years. Notably, the study of group determinants led Frobenius to laying out the foundation of representation theory. There is an extensive literature on the question when a variety in the projective space admits a determinantal representation. Without trying to give an exhaustive account of the references on this topic, we just mention \cite{D1}-\cite{D5}, \cite{KV}, \cite{V}, and also the monograph \cite{D} and references there. 

To avoid trivial redundancies, it is common to assume that at least one of the operators in the tuple is invertible, and, therefore, since $\sigma(A_1,..,A_n)=\sigma(A_n^{-1}A_1,...,A_n^{-1}A_{n-1}, I)$, \ $A_n$ can be taken to be the identity. In what follows, we will always assume that this is the case and write $\sigma(A_1,...,A_{n-1})$ instead of $\sigma(A_1,...,A_{n-1},I)$. Also, it was found to be useful to deal with the part of the  projective joint spectrum that lies in the chart $\{x_n\neq 0\}$. This part is called the \textit{proper projective joint spectrum} and is denoted by $\sigma_p(A_1,...,A_{n-1})$,
\begin{eqnarray*}
&\sigma_p(A_1,...,A_n) \\ = &\{ (x_1,...,x_n)\in \C^n:  x_1A_1+...+x_nA_n-I \ \mbox{is not invertible} \} .
\end{eqnarray*}

%It is a very simple observation that in general the spectral theorem does not hold for operator tuples, even if all the operators in the tuple are normal (or even self-adjoint). Thus, it became one of the fundamental questions of multivariable operator theory to find out under what conditions and to what degree results of classical spectral theory have multivariable analogs.  

In the last decade projective joint spectra of operator tuples have been intensively investigated 
(see \cite{BCY}-\cite{CST},\cite{DY},\cite{GoY}-\cite{GY}, \cite{MQW}, \cite{ST}-\cite{SYZ},\cite{Y}) from the following angle: what does the geometry of the projective joint spectrum tell us about the relations between operators in the tuple? 
%This, of course, might be considered as an analog of the inverse spectral problem mentioned above. 
As examples of recent results obtained in this direction we mention the following: \\
1) for a tuple of self-adjoint compact operators acting on a separable Hilbert space, the operators of the tuple mutually commute if and only if their projective joint spectrum, $\sigma(A_1,...,A_n)$, is a countable and locally finite union of projective hyperplanes, \cite{CSZ}, \cite{MQW}; \\
2) if the proper projective joint spectrum of a pair of self-adjoint matrices consists of certain type of "complex ellipses", then this pair is  decomposable in a direct sum of pairs of 2$\times$2 matrices and each of these pairs of 2$\times$2 matrices generates an irreducible representation of a dihedral group whose braid relation is determined by the corresponding ellipse, \cite{CST}. We just mention in passing that this result inspired the research in \cite{CST}  that characterized linear representations of finite non-special Coxeter groups in terms of joint spectra of images of Coxeter generators.

Coming back to Hadamard matrices, a straightforward check shows the following:

Let $A$ be the diagonal $n\times n$ matrix with $n$-th roots of unity appearing on the main diagonal in the increasing order of arguments, and $B=\frac{1}{n}F_n^*AF_n$ (of course, $A$ is the diagonalization of  $B$). Then
\begin{align}
&\sigma_p(A,B,AB)=\{ x^n+y^n+(-1)^{n-1}z^n=1\}\label{Fourier1} \\
&\sigma_p(A,B,AB,BA) =\{ x^n+y^n+(-1)^{n-1}(e^{2\pi i/n}z_1+z_2)^n=1\}\label{Fourier2}.	
\end{align}
Since this pair $A,B$ was generated by the Fourier matrix $F_n$, we call hypersurfaces in $\C^3$ and $\C^4$ given by the right hand sides of \eqref{Fourier1} and \eqref{Fourier2} \textit{the Fourier hypersurfaces}.

It easily follows from the results of \cite{CST}, section 5, that any pair of $2\times 2$ self-adjoint matrices $U,V$ with $\sigma_p(U,V,UV)$ 
%or $\sigma_p(U,V,UV,VU)$ 
being  given 
%respectively
 by 
\eqref{Fourier1} 
%or  \eqref{Fourier2} 
with $n=2$, is unitary equivalent to the pair $A,B$ generated by $F_2$, which may be considered as a spectral characterization of the $2\times 2$ Fourier matrix. 

Now, it is natural to ask whether the case $n=2$ is an exception, or it is a general fact and something similar is valid  for every $n$. It turned out that for $n>2$ the situation is more complex. 
\vspace{.2cm}

\begin{theorem}\label{Theorem2}
Let $A$ and $B$ be two $n\times n$ complex matrices such that $A$ is normal and $\parallel B\parallel=1$ with respect to the Euclidean norm on $\C^n$. Suppose that 
$$\sigma_p(A,B,AB)=\{x^n+y^n+(-1)^{n-1}z^n=1\} $$
Then
\begin{itemize}
\item[1).] Matrices $A$ and $B$ are unitary, and the spectra of $A$ and $B$, $\sigma(A)$ and $\sigma(B)$, consist of $n$-th roots of unity, each of multiplicity $1$.
\item[2).] 	If $e_0,...,e_{n-1}$ and $\zeta_0,...,\zeta_{n-1}$ are eigenbases for $A$ and $B$ respectively, such that $Ae_j=e^{2\pi ji/n}e_j$, \ $B\zeta_j=e^{2\pi ji/n}\zeta_j, \ j=0,...,n-1$, then $\sqrt{n}$ times the transition matrix from the basis $\{e_0,...,e_{n-1}\}$ to $\{\zeta_0,...,\zeta_{n-1}\}$ is a complex Hadamard matrix.
%\item[3).] 

If $n=3,4,5$ then the pair $(A,B)$ is unitary equivalent to the Hadamard pair of Fourier type.
\end{itemize}
	
\end{theorem}

As it was mentioned above, there is a complete description of complex Hadamard matrices of orders 3,4, and 5. This allows us to prove the following result. 

%n the following theorem we use these results to characterize pairs of matrices $(A,B)$ of orders 3, 4, and 5 whose projective joint spectrum $\sigma_p(A,B,AB)$ is the corresponding Fourier surface \eqref{Fourier1}. 

Let us  denote by $\widehat{B}_n$ and $\widehat{\widehat{B}}_n$ the following $n$-dimensional matrices:
\begin{align}
\widehat{B}_n=\frac{1}{n}F_n^*AF_n=\left [ \begin{array}{ccccc} 0 & 0 &... & 0 & 1\\ 1 & 0 & ... & 0 & 0 \\ .& .& . & . & . \\ 0 & 0 & ... & 1 & 0 \end{array} \right ] \label{B_hats}, \\
\widehat{\widehat{B}}_n= \frac{1}{n}F_nAF_n^*=\left [ \begin{array}{ccccc} 0 & 1 & 0 &... & 0 \\ .& .& .& .& . \\ 0 & 0 & 0& ... & 1 \\ 1 & 0 & 0 &... & 0     \end{array} \right ] \nonumber. \label{B_hats}
\end{align}

Of course,
\begin{equation}\label{hats}
\widehat{\widehat{B}}=\widehat{P}^*\widehat{B}\widehat{P}, 
\end{equation}
where $\widehat{P}$ is the matrix of the permutation 
$$ \left ( \begin{array}{ccccc} 0 & 1 &...& n-2 & n-1 \\
n-2 & n-3 &... & 0 & n-1 \end{array} \right ) $$

\begin{theorem}\label{Theorem3} Let $A$ and $B$ be two matrices satisfying the conditions of Theorem \ref{Theorem2}. 
\begin{itemize}
\item[a)] If $n=3$, then the pair $(A,B)$ is unitary equivalent 
\begin{align}\mbox{either to } \ \left ( \left [ \begin{array}{ccc} 1 & 0 & 0 \\ 0 & e^{2\pi i/3} & 0 \\ 0 & 0 & e^{4\pi i /3} \end{array} \right ] ,  \ \left [ \begin{array}{ccc} 0 & 1 & 0 \\ 0 & 0 & 1\\ 1 & 0 & 0 \end{array} \right ] \right ), \label{3a.1}\\ \mbox{or to} \ \left ( \left [ \begin{array}{ccc} 1 & 0 & 0 \\ 0 & e^{2\pi i/3} & 0 \\ 0 & 0 & e^{4\pi i /3} \end{array} \right ] ,  \ \left [ \begin{array}{ccc} 0 & 0 & 1 \\ 1 & 0 & 0\\ 0 & 1 & 0 \end{array} \right ] \right ) \label{3a.2}\end{align}. 
\item[b)] If $n=4$, or $n=5$,  there exists a permutation $P\in{\mathcal P}_n$, the group of permutations of order $n$, such that the pair $(A,B)$ is unitary equivalent to either $(P^*AP,\widehat{B}_n)$, or $(\widehat{B}_n,P^*AP)$.

\end{itemize}
\end{theorem}

\textbf{Remark}. We would like to mention that Theorem \ref{Theorem3} provides additional counterexamples to the question whether the joint spectrum of images of generators of a finite group under a linear representation determines the representation up to an equivalence. This question was motivated by the result 
in \cite{CST}, stating that the answer was positive for finite non-special Coxeter groups. It
 was discussed during the Banff conference on Multivariable Operator Theory and Representation Theory in April 2019. 
The first counterexample showing that in general the answer is negative, was produced by  I. Klep and J. Vol\v{c}i\v{c} \cite{KV}. Now we can provide additional ones. For instance, the group $G$ generated by $g_1$ and $g_2$ satisfying the relations
$$g_1^3=g_2^3=(g_1g_2)^3=(g_1g_2^2)^3=e$$
is a group of order 27. Consider two representations $\rho_1$ and $\rho_2$ of $G$, the first being generated by \eqref{3a.1},  $\rho_1(g_1)=A, \rho_1(g_2)=\widehat{\widehat{B}}_3$, and the second - by \eqref{3a.2}, $\rho_1(g_1)=A, \rho_1(g_2)=\widehat{B}_3$. They are clearly not equivalent, but 
$$\sigma_p(\rho_1(g_1),\rho_1(g_2))=\sigma_p(\rho_2(g_1),\rho_2(g_2))=\{x^3+y^3=1\}. $$ 
More examples could be produced using matrices from section b) of Theorem \ref{Theorem3}.

\vspace{.2cm}

 Theorem \ref{Theorem3} shows that for $n>2$ there is no rigidity, that is the joint spectrum of $A,B$, and $AB$ does not determine the pair up to a unitary equivalence. Our next result shows that such rigidity can be obtained by adding $BA$ to the tuple. It also can be viewed as a spectral characterization of $F_n$. We give here an infinite dimensional version and use the standard notation, $\Delta_n(x, \rho)$, for the polydisk in $\C^n$  of radius $\rho$ and centered at $x$,
$$\Delta_n(x,\rho)=\{w= (w_1,...,w_n)\in \C^n: \ |w_j-x_j|<\rho, \ j=1,2,...,n \}.$$

\begin{theorem}\label{Theorem4}
Suppose that $A$ is a normal operator and $B$ is an operator of norm one, both acting on a Hilbert space $H$. Suppose that  there exists an $\epsilon>0$ such that the joint spectrum \ $\sigma_p(A,B,AB,BA)$, satisfies the following conditions:
\begin{itemize}
\item[1)] \begin{align*}
 &\sigma_p(A,B,AB,BA)\cap \Delta_4 ( \tilde{\tau}_{mj},\epsilon) \\ 
 =& \{x^n+y^n+(-1)^{n-1}(e^{2\pi i/n}z_1+z_2)^n=1 \} \cap \Delta_4  ( \tilde{\tau}_{mj},\epsilon), \\ & m=1,2, \ j=0,...,n-1,
 \end{align*}	
where  $$\tilde{\tau}_{1j}=(e^{2\pi j i/n},0,0,0), \ \tilde{\tau}_{2j}=(0,e^{2\pi j i/n},0,0), j=0,...,n-1$$.
\item[2).] The multiplicity of each point of 
$\{x^n+y^n+(-1)^{n-1}(e^{2\pi i/n}z_1+z_2)^n=1 \} \cap \Delta_4 ( \tilde{\tau}_{m,j},\epsilon), \ m=1,2, \ j=0,...,n-1,$ in \newline
$\sigma_p(A,B,AB,BA)$ is equal to one.
\end{itemize}

Then 
\begin{itemize}
\item[1)] There is a subspace $L\subset H$ of dimension $n$ invariant under the action of $A$ and $B$.
\item[2)] The restriction of the pair $(A,B)$ to the subspace $L$  is unitary equivalent to

\begin{equation}\label{Four}
\left [ \begin{array}{cccc} 1 & 0 & ...& 0\\ 0 & e^{2\pi i/n } & ... & 0 \\ . & . & . & . \\ 0 & 0 & ... & e^{2\pi (n-1) i/n} \end{array}\right ], \ \left [ \begin{array}{ccccc} 0 & 0 & 0 &... &  1 \\ 1 & 0 &0& ... &0 \\ 0 & 1 & 0 &... & 0 \\. & . &. & . & . \\ 0 & 0 & ... & 1 & 0 \end{array} \right ] . \end{equation}
and $\sqrt{n}$ times the transition matrix of Theorem \ref{Theorem2} is in the form $\Lambda_1F_n\Lambda_2$, where $\Lambda_{1,2}$ are diagonal matrices with unimodular entries on the main diagonal.
\end{itemize}
\end{theorem}

\vspace{.2cm}

Of course, passing from the bases  $e_0,...,e_{n-1}$ and $\zeta_0,...,\zeta_{n-1}$ in Theorem \ref{Theorem4}  to $e^{i\theta_0}e_0,...,e^{i\theta_{n-1}}e_{n-1}$ and $e^{i\varphi_0}\zeta_0,...,e^{i\varphi_{n-1}}\zeta_{n-1}$ respectively with properly chosen arguments $\theta_j$ and $\varphi_j$, makes it possible to get rid of the matrices $\Lambda_{1,2}$ and have $\frac{1}{\sqrt{n}}F_n$ as the transition matrix.

\vspace{.2cm}

As a Corollary to this result we obtain the following statement:

\begin{corollary}\label{cor}
A complex Hadamard matrix $h$ of order $n$ is in the form $\Lambda_1F_n\Lambda_2$, where $\Lambda_{1,2}$ are diagonal matrices with unimodular entries on the main diagonal,
%Fourier matrix $F_n$ 
if and only if 
$$\sigma_p(A,B,AB,BA)=\{x^n+y^n+(-1)^{n-1}(e^{2\pi i/n}z_1+z_2)^n=1 \}, $$
 ($A$ is the first (diagonal) matrix in \eqref{Four}, and $B=\frac{1}{n}h^*Ah$).
\end{corollary}

%Of course, passing from the bases  $e_0,...,e_{n-1}$ and $\zeta_0,...,\zeta_{n-1}$ in Theorem \ref{Theorem4} and Corollary \ref{cor} to $e^{i\theta_0}e_0,...,e^{i\theta_{n-1}}e_{n-1}$ and $e^{i\varphi_0}\zeta_0,...,e^{i\varphi_{n-1}}\zeta_{n-1}$ respectively with properly chosen arguments $\theta_j$ and $\varphi_j$, makes it possible to get rid of the matrices $\Lambda_{1,2}$ and have $\frac{1}{\sqrt{n}}F_n$ as the transition matrix.

The proofs of Theorems \ref{Theorem2}, \ref{Theorem3}, and \ref{Theorem4} are based on local analysis near a regular point of the projective joint spectrum of a tuple of operators acting on a Hilbert space $H$. In \cite{ST} this kind of analysis was produced for two self-adjoint operators. Here we need it for a tuple of size bigger than 2 and without the condition that operators are self-adjoint. The difference in argument is not substantial, and the constructions are mostly similar, but to make our presentation self-contained we included the necessary details. 

This local analysis concludes with the following theorem, which is also one of our main results and its proof establishes the technique used in our investigation.

\begin{theorem}\label{Theorem1}
Let $A$ and $B$ be two bounded operators acting on a Hilbert space $H$ and satisfying the following conditions: $A$ is normal and $\parallel B\parallel =1$. Suppose that $n \in {\mathbb N}$ and that there exists an $\epsilon >0$ such that the proper projective joint spectrum of operators $A,B$, and $AB$, \ $\sigma_p(A,B,AB)$, satisfies the following conditions:
\begin{itemize}
\item[a)] \begin{align*}
 &\sigma_p(A,B,AB)\cap \Delta_3  ( \tau_{mj},\epsilon) \\ 
 =& \{x^n+y^n+(-1)^{n-1}z^n=1 \} \cap \Delta_3  ( \tau_{m,j},\epsilon), \ m=1,2, \ j=0,...,n-1,
 \end{align*}	
where   $\tau_{1j}=(e^{2\pi j i/n},0,0), \ \tau_{2j}=(0,e^{2\pi j i/n},0), j=0,...,n-1$.
\item[b)] Each point of $ \sigma_p(A,B,AB)\cap \Delta_3  ( \tau_{mj},\epsilon)$ has multiplicity one.
%\item[2).] The multiplicity of each point of 
%$\{x^n+y^n+(-1)^{n-1}z^n=1 \} \cap \Delta  ( \tau_{m,j},\epsilon), \ m=1,2, \ j=0,...,n-1,$ in $\sigma_p(A,B,AB)$ is equal to $k$.
\end{itemize}

Then 

\begin{itemize}
\item[1)] The $n$-dimensional subspace $L$ of $H$ spanned by $A$-eigenvectors with eigenvalues $e^{2\pi m i/n}, \ m=0,...,n-1$ is invariant under the action of $B$.

\item[2)] The restriction of $B$ to $L$ is unitary ( and, obviously, this is true for $A$).
\end{itemize}
\end{theorem}

The structure of this paper is as follows. In section \ref{hypersurfaces} we prove a generalization of a result in \cite{ST} proved for operator pairs, to operator tuples of arbitrary length. This is our main tool of local spectral analysis. In section \ref{thm1} we use the results of section \ref{hypersurfaces} to prove Theorem \ref{Theorem1}. 
%Operator relations established in this proof are repeatedly used in establishing the results of the paper. 
Section \ref{thm2_and_thm3} is devoted to Theorems \ref{Theorem2} and \ref{Theorem3}. The proof of the former relies on operator relations established in the proof of Theorem \ref{Theorem1}. In the proof of the latter,
the case $n=3$ is treated directly by using the same operator relations. We then prove that for every $n$ certain pairs of permutations $(P_1,P_2)$ in \eqref{hadamard}{ determine pairs $(A,B)$ with projective spectrum being the Fourier surface \eqref{Fourier1}.  We further use SAGEMATH software package to show that for $n=4,5$ all other pairs of permutations do not. It turned out that for $n=4$ in order for the joint spectrum to be the Fourier surface in $\C^3$, the transition matrix must be similar to $F_4$ as well. In section \ref{thm4} we prove Theorem \ref{Theorem4} and Corollary \ref{cor}. Finally, section \ref{last} contains algorithms used for establishing part b) of Theorem \ref{Theorem3}.

\section{Spectral algebraic hypersurfaces }\label{hypersurfaces}

Theorem 4.2 in \cite{ST} shows that the presence of an algebraic curve in the joint spectrum of two self-adjoint operators implies that certain relations between these operators hold. Using practically the same technique it is possible to get a similar result for operator tuples of size bigger than two. As mentioned in the introduction, to make our presentation self-contained we include the details in this section.

\vspace{.2cm}

\noindent Let $A_1,...,A_n$ be bounded operators acting on a Hilbert space $H$, and let $\lambda \neq 0$ be an isolated spectral point of finite multiplicity in $\sigma(A_1)$. Then $(1/\lambda, 0,...,0)\in \sigma_p(A_1,...,A_n)$. It was shown in \cite{ST} that in such case $\sigma_p(A_1,...,A_n)$ is an analytic set in a neighborhood of this point. Assume that $(1/\lambda, 0,...,0)$ is a regular point of this analytic set, and has multiplicity one.
We also assume that $x_1$-axis is not tangent to $\sigma_p(A_1,...,A_n)$ at $(1/\lambda,0...,0)$. 
Then
\begin{itemize}
\item[1)] for every $(x_1,...,x_n)\in \sigma_p(A_1,...,A_n)$ such that $x_1$ close to $1/\lambda$	and $x_2,...,x_n$ close to zero, the pencil $x_1A_1+...+x_nA_n$ has one as an isolated eigenvalue of multiplicity one;
\item[2)] for every $ (x_1,...,x_n)$ close to $(1/\lambda,0...,0)$ the line through the origin and $x=(x_1,...,x_n)$ has only one point of intersection with $\sigma_p(A_1,...,A_n)$ which is close to $(1/\lambda,0...,0)$ 

\end{itemize}
  
%\vspace{.2cm}

\noindent Of course, 1), 2) impliy that there exists a $\rho>0$ such that for every $x=(x_1,...,x_n)$ close to 
$(1/\lambda,0,...,0)$
\begin{equation}\label{projection}
P(x)=\frac{1}{2\pi i} \int_{|w-1|=\rho } (w-A(x))^{-1}dw
\end{equation}
is a rank one projection on the eigenspace of the pencil $A(x)=x_1A_1+...+x_nA_n$ corresponding to the only eigenvalue of $A(x)$ close to one.

\vspace{.2cm}

\noindent Now suppose that in a neighborhood of $(1/\lambda,0...,0)$ the joint spectrum $\sigma_p(A_1,...,A_n)$ is an algebraic set and is given by $\{ {\mathcal M} (x_1,...,x_n)=0\}$, where ${\mathcal M}$ is a polynomial of degree $k$ with homogeneous decomposition
$$
{\mathcal M}(x_1,...,x_n)=\sum_{j=0}^k M_j(x_1,...,x_n), \ M_0=-1 .
$$
Fix $x=(x_1,...,x_n)$ close to $(1/\lambda,0,...,0)$. Then the polynomial equation in $\tau  \in \C$  
\begin{equation}\label{eigenvalues1}
{\mathcal M}(\tau x)=\sum_{j=0}^k \tau^j M_j(x)=0 
\end{equation}
has only one root close to one. Let $\mu_1,...,\mu_k$ be the reciprocals of the roots of \eqref{eigenvalues1}. They satisfy the equation
\begin{equation}\label{eigenvalues2} 
\mu^k-M_1(x)\mu ^{k-1}-...-M_k(x)=0.,
\end{equation}
and, again, \eqref{eigenvalues2} has only one root near one. 
%and \eqref{projection} gives the projection on the corresponding one-dimensional subspace. 
Thus, $\forall \zeta \in H$
$$
(A(x)^k-M_1(x)A(x)^{k-1}-...-M_k(x))P(x)\zeta=0,	
$$
and, therefore,
$$ (A(x)^k-M_1(x)A(x)^{k-1}-...-M_k(x))P(x)=0. $$
Considering that
$$A(x)^jP(x)=\frac{1}{2\pi i} \int_{|w-1|=\rho} w^j(w-A(x))^{-1}dw, $$
we obtain
\begin{equation}\label{integral_relation}
\frac{1}{2\pi i} \int_{|w-1|=\rho}\left [ w^k-	\sum_{j=1}^k M_j(x)w^{k-j} \right ] (w-A(x))^{-1}dw=0.
\end{equation}

Now,
\begin{align*}
&(w-A(x))^{-1}=\left [ (w-\frac{1}{\lambda}A_1) - (x_1-\frac{1}{\lambda})A_1-x_2A_2-...-x_nA_n) \right ] ^{-1}\\
&=\left ( w- \frac{1}{\lambda}A_1 \right )^{-1} \left [ I - \left ( (x_1-\frac{1}{\lambda})A_1+x_2A_2+...+x_nA_n \right ) \left (w- \frac{1}{\lambda}A_1 \right )^{-1} \right ] ^{-1} \\
&=\left ( w-\frac{1}{\lambda}A_1 \right ) ^{-1} \sum_{j=0}^\infty. \left [\left ( (x_1-\frac{1}{\lambda})A_1+x_2A_2+...+x_nA_n \right ) \left (w- \frac{1}{\lambda}A_1 \right )^{-1} \right ] ^j .
\end{align*}

%the numbering of this equation overlaps onto the equation so I pushed it to the right with \quad
Set $x_1=\frac{1}{\lambda}$, then
\begin{eqnarray}
&\quad \left ( w-A(\frac{1}{\lambda},x_2,...,x_n) \right ) ^{-1} \nonumber \\
&\quad =\left ( w-\frac{1}{\lambda}A_1 \right ) ^{-1} \sum_{j=0}^\infty \left [ \left ( x_2A_2+...+x_nA_n \right ) \left ( w- \frac{1}{\lambda}A_1 \right )^{-1} \right ] ^j .\label{decomposition}	
\end{eqnarray}

It was shown in \cite{ST} that if 
$$A_1=\lambda P_0 +\int_{\sigma( A_1)\setminus \{\lambda \}} zdE(z)$$
is the spectral resolution of $A_1$, and 
\begin{equation}\label{T}
T=\int_{\sigma(A_1)\setminus \{\lambda\}} \left ( \frac{\lambda}{z-\lambda}\right )dE(z),
\end{equation}
then 
\begin{equation}\label{inverse}
\left ( w-\frac{1}{\lambda}	A_1 \right ) ^{-1}= \frac{1}{w-1}P_0-\sum_{m=0}^\infty T^{m+1}(w-1)^m.
\end{equation}

Thus, the relation \eqref{decomposition} can be written as

\begin{align}
\left ( w-A \left (\frac{1}{\lambda},x_2,...,x_n\right )\right )^{-1} 
=\left ( \frac{1}{w-1}P_0-\sum_{m=0}^\infty T^{m+1}(w-1)^m \right ) \nonumber \\
\times \sum_{j=0}^\infty \left [ (x_2A_2+...+x_nA_n) \left ( \frac{1}{w-1}P_0-\sum_{m=0}^\infty T^{m+1}(w-1)^m \right ) \right ] ^j . \label{inverse3}
\end{align}

Write 
\begin{align*}
\Psi (w;x)=\left [ w^k-	\sum_{j=1}^k M_j(x)w^{k-j} \right ] \left ( \frac{1}{w-1}P_0-\sum_{m=0}^\infty T^{m+1}(w-1)^m \right ) \\
\times \sum_{j=0}^\infty \left [ (x_2A_2+...+x_nA_n) \left ( \frac{1}{w-1}P_0-\sum_{m=0}^\infty T^{m+1}(w-1)^m \right ) \right ] ^j .
\end{align*}
This operator-valued function $\Psi$ is analytic in $x_2,...,x_n$ in a neighborhood of zero and meromorphic in $w$ in a neighborhood of $w=1$. If we write down the Taylor decomposition of $\Psi$ with respect to
$x_2,...,x_n$,
$$\Psi(w;x_2,...,x_n)= \sum_{m_2,...,m_n} x_2^{m_2}...x_n^{m_n} \psi_{m_2,...,m_n}(w), $$ 
equation \eqref{integral_relation} implies that for all $m_2,...,m_n$
\begin{equation}\label{residue1}
Res \ \psi_{m_2,...,m_n} \left | _{w=1} \right. =0 . 	
\end{equation}
It was shown in \cite{ST} that for $n=2$  the last relation implies that all functions $\psi$ are analytic in a neighborhood of the origin. A similar proof shows that the same is true for arbitrary size tuples.

Relation \eqref{residue1} depends on the coefficients of the polynomial ${\mathcal M}$, so we will need an explicit formula. Write
$$M_j(x)= \sum_{k_1+k_2+...+k_n=j}d_{k_1k_2...k_n}x_1^{k_1}...x_n^{k_n}, $$
Let $m=(m_2,...,m_n)$ and $l=(l_2,...,l_n)$ be two multi index  sets. We write  $l\prec m$, if for every $2\leq j \leq n , \ \ l_j\leq m_j$. If $l\prec m$, we write $c(l,m)=\sum_{j=2}^n (m_j-l_j), \ \  I(l,m)=\{j: l_j<m_j\}$ and define 
\begin{align*}& \Lambda (l,m)=\{ (r_1,..., r_{c(l,m)}): r_s\in I(l,m) \ s=1,...,c(l,m), \\ &\mbox{and each index in $j\in I(l,m)$ occurs in $(r_1,...,r_{c(l,m)}) \ (m_j-l_j)$ times} \}.
\end{align*}

Then
\begin{align}
&\psi_{m_2,...,m_n}(w) \nonumber \\=&\left \{ \left ( w^k-\sum_{r=1}^k \frac{d_{r,0...0}}{\lambda^r}w^{k-r} \right ) \left (\frac{1}{w-1}P_0-\sum_{m=0}^\infty T^{m+1}(w-1)^m\right ) \right.\nonumber \\
%\left ( w-\frac{A_1}{\lambda} \right )^{-1}  \right. 
&\times \left.\sum_{(j_1,...,j_{m_2+...+m	_n})} \prod_{s=1}^{ m_1+...+m_n} A_{j_s}\left ( \frac{1}{w-1}P_0-\sum_{m=0}^\infty T^{m+1}(w-1)^m \right ) ^{-1} \right \} \label{residue_in_coefficients}\\
&- \left (\frac{1}{w-1}P_0-\sum_{m=0}^\infty T^{m+1}(w-1)^m\right )
%\left ( w-\frac{A_1}{\lambda} \right ) ^{-1} 
 \nonumber \\
&\left \{ \left [ \sum_{ (l_2,...,l_n)\prec (m_2,...,m_n)} \ \ \sum_{ \Lambda(l,m)}\left ( \sum_{s=1}^k \frac{d_{s,l_2,...,l_n}}{\lambda^s}w^{k-s-|l|} \right )  \right. \right. \nonumber \\  
 & \left. \left.\times \prod_{t=1}^{c(j,m)} A_{r_t} \left (\frac{1}{w-1}P_0-\sum_{m=0}^\infty T^{m+1}(w-1)^m\right )\right ] +d_{m_1...m_n}I
  \right \} ,\nonumber 
 \end{align}
where the second sum in the righthand side is taken over all multi index sets $(j_1,...,j_{m_1+...+m_n})$ where $1$ occurs $m_1$ times, 2 - $m_2$ times,..., $n$ - $m_n$ times.

\section{Proof of Theorem \ref{Theorem1}}\label{thm1}

We now will apply local analysis given by relations \eqref{inverse} - \eqref{residue_in_coefficients} to prove statement 1) of Theorem \ref{Theorem1}. We begin with the following proposition. Let, again, $P_0$ be the orthogonal projection on the eigensubspace of $A$ corresponding to the eigenvalue one.

\begin{proposition}\label{first-fact}
\begin{align}
&P_0(BT)^kBP_0=0, \ k=0,...,n-2,\label{moments.1.1} \\ &P_0(BT)^{n-1}BP_0=(-1)^{n-1}\frac{1}{n}P_0, \label{moments.1} \\
&P_0(ABT)^kABP_0=0, \ k=0,...,n-2, \label{moments.2.1}\\ &P_0(ABT)^{n-1}ABP_0=\frac{1}{n}P_0. \label{moments.2}
\end{align}
where  $T$	is given by \eqref{T}.
\end{proposition}

\begin{proof} 
In our case we have three operators $A,B$, and $AB$ with $A$ playing the role of $A_1$ in the previous section. We have by \eqref{residue_in_coefficients}
\begin{align*}\psi_{1,0}(w)=(w^n-1)\left ( \frac{P_0}{w-1}-\sum_{j=0}^\infty (w-1)^jT^{j+1}\right ) \\  \times \left (\frac{BP_0}{w-1}-\sum_{j=0}^\infty (w-1)^jBT^{j+1} \right ), \end{align*}	
so that 
\begin{equation}\label{first_moment1}Res_{w=1}(\psi_{1,0}(w))=n P_0BP_0=0. \end{equation}
Similarly,
\begin{align*}
&\psi_{2,0}(w)=(w^n-1)\left ( \frac{P_0}{w-1}-\sum_{j=0}^\infty (w-1)^jT^{j+1}\right )   \\
&\times \left (\frac{BP_0}{w-1}-\sum_{j=0}^\infty (w-1)^jBT^{j+1} \right )	\left (\frac{BP_0}{w-1}-\sum_{j=0}^\infty (w-1)^jBT^{j+1} \right ),
\end{align*}
which yields
\begin{align*}
&Res_{w=1} \psi_{2,0}(w)=	n \left (- P_0BP_0BT-P_0BTBP_0- TBP_0BP_0 \right )  \\
&+\frac{n(n-1)}{2}P_0BP_0BP_0.
\end{align*}
Now, \eqref{first_moment1} implies
\begin{equation}\label{second_moment1}
P_0BTBP_0=0.	
\end{equation}
In general,
\begin{align*}\psi_{k,0}(w)=(w^n-1)\left ( \frac{P_0}{w-1}-\sum_{j=0}^\infty (w-1)^jT^{j+1}\right ) \\
\times \left (\frac{BP_0}{w-1}-\sum_{j=0}^\infty (w-1)^jBT^{j+1} \right )^k. \end{align*}
We claim that for $k\leq n$ powers of $T$ higher than one do not appear in the expression of the residue at $w=1$. We prove it using induction in $k$. For $k=1,2$, \eqref{first_moment1} and \eqref{second_moment1} show that it is true. Suppose that it is true for $k\leq m$  and consider $\psi_{m+1,0}(w)$. Every power of $T$ higher than one comes with the factor $(w-1)$ raised  to a power one, or higher. Thus, to contribute to the residue at $w=1$, a monomial in the decomposition of the above product representing $\psi_{m+1,0}$ must contain at least as many $BP_0$ as the number of occurrences of $T^s$ with $s>1$ plus one (we multiply the first parentheses by $w-1$ coming from the decomposition $w^n-1=(w-1)(w^{n-1}+...+1)$, so that there is no negative power of $(w-1)$ there). It follows that there are two $BP_0$ in this monomial such that there are no $T$ raised to a power higher than one between them, so this monomial looks like
$$...B\underbrace{P_0BT...BTBP_0}... $$
By the induction assumption the underbraced part vanishes, and we proved that monomials containing powers of $T$ higher than one do not contribute to the residue.

It was shown in \cite{ST} that for every $k$ the condition $Res_{w=1} \psi_{m,0}=0$ for $m\leq k$ implies that $\psi_{k,0}$ is holomorphic in a neighborhood of $w=1$. Now, it is easy to derive inductively that $Res_{w=1} \psi_{k,0}(w)=P_0\underbrace{BTBT...BT}_{k-1}BP_0=0$ for $k\leq n-1$, which proves  \eqref{moments.1.1}. Also, our previous argument, \eqref{residue_in_coefficients}, and the fact that variable $y$  appears in the equation $\{x^n+y^n+ (-1)^{n-1}z^n=1\}$ raised to the power $n$ only, immediately shows that
$$ P_0\underbrace{BT...BT}_{n-1}BP_0=\frac{(-1)^{n-1}}{n}P_0.$$
 The proofs of \eqref{moments.2.1} and \eqref{moments.2} go along the same lines. We are done.

\end{proof}

\vspace{.2cm}

Since the polynomial $x^n+y^n+(-1)^{n-1}z^n-1$ contains no monomials other than $x^n,y^n$, and $z^n$, an argument similar to the one of Proposition \ref{first-fact} yields the following result.

\begin{proposition}\label{second-fact}
Let $m\leq n-1$, and  $r=(r_1,...r_k)$ satisfy $1\leq r_1<r_2<...<r_k\leq m$. Define $C(r)$ by
$$C_m(r_1,...,r_k)=\prod_{j=1}^m S_j$$ where 
$$S_s=\left \{ \begin{array}{ccc}
 ABT & \mbox{if} & s=r_1,...,r_k, \\ BT & \mbox{if} & s\neq r_1,...,r_k \end{array} \right. , \ s=1,...,m-1, \ S_m=\left \{ \begin{array}{ccc}	AB & \mbox{if} & r_k=m \\ B &\mbox{if} & r_k\neq m \end{array} \right.
$$
Further, let 
$${\mathcal C}(k,m)= \sum_{1\leq r_1<r_2<...<r_k\leq m } C_m(r_1,...,r_k).$$ 

Then 
\begin{align}
&P_0{\mathcal C}(k,m)P_0=0, \ \mbox{if} \ m=1,...,n-1, \ \mbox{or} \ m=n, \ k<n \label{second-fact.1}\\
&P_0{\mathcal C}(n,n)P_0=\frac{1}{n}P_0. 	\label{second-fact.2}
\end{align}
 
\end{proposition}

\vspace{.2cm}

Similarly, we introduce
$$D_m(r_1,...,r_k)=\prod_{j=1}^m D_j, $$
where
$$D_s=\left \{ \begin{array}{ccc}
TAB & \mbox{if} & s=r_1,...,r_k, \\ TB & \mbox{if} & s\neq r_1,...,r_k \end{array} \right. , \ s=2,...,m, \ D_1=\left \{ \begin{array}{ccc}	AB & \mbox{if} & r_1=1 \\ B &\mbox{if} & r_1\neq 1 \end{array} \right.
$$
and write 
$${\mathcal D}(k,m)= \sum_{1\leq r_1<r_2<...<r_k\leq m } D_m(r_1,...,r_k).$$ 

Since $P_0A=AP_0=P_0$, the following Corollary follows immediately from Proposition \ref{second-fact}.

\begin{corollary}\label{D}
\begin{align}
&P_0{\mathcal D}(k,m)P_0=0, \ \mbox{if} \ m=1,...,n-1, \ \mbox{or} \ m=n, \ k<n \label{second-fact.1}\\
&P_0{\mathcal D}(n,n)P_0=\frac{1}{n}P_0. 	\label{second-fact.3} \\
&P_0{\mathcal D}(0,n)P_0=\frac{(-1)^{n-1}}{n}P_0
\end{align}	
\end{corollary}

\begin{proposition}\label{eigenvector}
	\begin{align} 
	&P_0B^kP_0=0, \ k=1,...,n-1 \label{Powers_of_B.1} \\ &P_0B^nP_0=P_0 \label{powers_of_B} \
	\end{align}

\end{proposition}

\begin{proof} 

Observe that
\begin{align}
&AT=TA=\int_{\sigma(A)\setminus \{ 1\} } \frac{z}{z-1}dE(z)=\int_{\sigma(A)\setminus \{1\}}dE(z)+\int_{\sigma(A)\setminus \{1\}} \frac{dE(z)}{z-1} \nonumber \\
&= I-P_0+T. \label{A_and_T}
\end{align}
Therefore, 
\begin{equation}\label{TAB}
TAB=B-P_0B+TB, \end{equation}
so that
\begin{equation}\label{B_through_A_and_T}
B=TAB+P_0B-TB. 
\end{equation}

Thus,
$$P_0B^kP_0=P_0 (TAB+P_0B-TB)^kP_0. $$
Since $P_0T=0 \ \mbox{and} \ P_0^2=P_0$, we have
$$P_0B^kP_0=P_0 B(TAB+P_0B-TB)^{k-1}P_0.$$
Making all multiplications we obtain
$$P_0B^kP_0=P_0B\left (\sum  \prod_{j=1}^{k-1} S_j\right )P_0 ,$$
where each $S_j$ is either $(TAB-TB$), or $P_0B$. We now rewrite this sum according to the last position where $S_j=P_0B$. Since $P_0BP_0=0$, we have

\begin{align*} 
&P_0B^kP_0= P_0B  \left \{ (TAB-TB)^{k-1} \right. \\
 &+\left.\left [\sum_{j=0}^{k-2} \left ( \sum  \prod_{t=1}^{k-1-j}S_t\right )  P_0B(TAB-TB)^j\right ] \right \}P_0,
\end{align*}
where each $S_t$ in the right hand side is, as above, either $(TAB-TB)$ or $P_0B$ and there are no two $P_0B$ next to each other. Thus,
\begin{align*}
%\right. \\ &\left.+\sum_{j=1}^{k-1}(TAB+P_0B-TB)^{k-1-j}P_0B(TAB-TB)^j \right ]P_0 \\
&P_0B^kP_0=P_0B(TAB-TB)^{k-1}P_0 \\ &+P_0B\left [\sum_{j=0}^{k-2} \left ( \sum  \prod_{t=1}^{k-1-j}S_t\right )  P_0B(TAB-TB)^j\right ] P_0,
%\sum_{j=1}^{k-1}(TAB+P_0B-TB)^{k-1-j}P_0B P_0.  
\end{align*}
and here $S_1=(TAB-TB)$.

Note that by Corollary \ref{D}
$$P_0B(TAB-TB)^jP_0= \sum_{s=0}^j (-1)^{j-s}P_0{\mathcal D}(s,j)P_0=0 \ \mbox{for} \ j<n.$$
%Thus,
%\begin{equation}\label{without_P_0B}
%P_0B^kP_0=P_0(TAB-TB)^kP_0=\sum_{j=0}^k (-1)^{k-j}P_0{\mathcal D}(j,k)P_0=0, \ k<n, 
%\end{equation}
which proves \eqref{Powers_of_B.1}. To establish \eqref{powers_of_B}  we remark that  $P_0A=AP_0=P_0$, so that

\begin{align}
&P_0{\mathcal C}(k,m)P_0=P_0BT\tilde{{\mathcal C}}(k-1,m)P_0+P_0\tilde{{\mathcal C}}(k,m)P_0, \ 2\leq k\leq  m-1 \label{with_A.1} \\
&P_0{\mathcal C}(1,m)P_0=P_0(BT)^{m-1}BP_0+P_0\tilde{{\mathcal C}}(1,m)P_0, \label{with_A.2} \\
&P_0{\mathcal C}(m,m)P_0=P_0\tilde{{\mathcal C}}(m-1,m)P_0, \label{with_A.3}
\end{align}
where
$$\tilde{{\mathcal C}}(k,m)=\sum_{2\leq r_1<r_2<...<r_k\leq m } C_m(r_1,...,r_k). $$
 Relations \eqref{second-fact.1} and \eqref{with_A.1} imply
 \begin{equation} \label{different_A}
 P_0\tilde{{\mathcal C}}(k-1,m)P_0=-P_0\tilde{{\mathcal C}}(k,m)P_0,	 \ 2\leq k\leq  n-1. 
 \end{equation}

Since \eqref{second-fact.2} and \eqref{with_A.3}	imply
\begin{equation}\label{n-th degree}
	P_0\tilde{{\mathcal C}}(n-1,n)P_0=\frac{1}{n}P_0, 
\end{equation}
we obtain
\begin{equation}\label{alternation}
P_0\tilde{{\mathcal C}}	(k,n)P_0= \frac{(-1)^{n-k-1}}{n}P_0.
\end{equation}

We now have by what was proved above

$$ P_0B^nP_0=P_0B(TAB-TB)^{n-1}P_0=\sum_{k=0}^{n-1} (-1)^{n-k-1}P_0\tilde{{\mathcal C}}(k,n)P_0=P_0.$$

 We are done.
\end{proof}

\vspace{.1cm}

\begin{corollary}\label{orthogonal.1}
 Let $e_0$ be a unit eigenvector of $A$ with eigenvalue one. Then
 \begin{equation}\label{invariant.1}
 \langle  B^ne_0,e_0\rangle =1.	
 \end{equation}	
\end{corollary}

The result follows immediately from \eqref{powers_of_B}

\vspace{.2cm}

We are ready to finish the proof of Theorem \ref{Theorem1}. 

Since $\parallel B\parallel =1$, Proposition \ref{eigenvector} shows that the eigenvector of $A$ with eigenvalue one is an eigenvector of $B^n$ with the same eigenvalue one. 

It is easy to see that for every $1\leq m\leq n-1$ the joint spectrum $\sigma_p((e^{2\pi im/n}A),B,(e^{2\pi im/n}A)B)$ contains the same algebraic surface $\{x^n+y^n+(-1)^{n-1}z^n=1\}$. Indeed, 
\begin{align*}
(x,y,z)\in \sigma_p((e^{2\pi im/n}A),B,(e^{2\pi im/n}A)B) \\ \Longleftrightarrow  (e^{2\pi im/n}x,y,e^{2\pi im/n}z)\in \sigma_p(A,B,AB).
\end{align*} 
Application of the above argument shows that
the one-eigenvector on $e^{2\pi im/n}A$ is a one-eigenvector of $B^n$. Thus, $B^n$ turns into identity on the $n$-dimensional subspace $L$ spanned by eigenvectors of $A$ with eigenvalues $n$-th roots of unity. 

Since the component $\{x^n+y^n+(-1)^{n-1}z^n=1\}$ has multiplicity one, and since for every $m=0,...,n-1$ 
\begin{align*}\sigma_p(A,B,AB)\cap \Delta_3(\tau_{2m},\epsilon)= \{x^n+y^n+(-1)^{n-1}z^n=1\} \cap \Delta_3(\tau_{2m}, \epsilon), \end{align*}
every $n$-th root of unity is an isolated spectral point of $B$ of multiplicity one. Let $\gamma_j, \ j=0,...,n-1$ be a circle in  $\C$ centered at $e^{2\pi ij/n}, \ j=0,...,n-1$, that  does not contain other spectral points of $B$, and 
$$ P=\frac{1}{2\pi i}\int_{\gamma_0+\gamma_2+...+\gamma_{n-1}} (w-B)^{-1}dw.$$
The range of the projection $P$ is invariant under $B$, has dimension $n$,  and, the spectral mapping theorem  (cf \cite{La}) implies that it contains all 1-eigennvectors of $B^n$. Thus, $L=Range (P)$ is invariant under the action of both $A$ and $B$. This finishes the proof of the first part of Theorem \ref{Theorem1}.

To prove that the restriction of $B$ to the subspace $L$ is unitary, we observe that, since no power of a Jordan cell of dimension greater than one is diagonal, $B$ is diagonalizable on $L$, each eigenvalue of the restriction $\left.B\right |_L$ is an $n$-th root of unity, and multiplicity of each of them is one. Now, the result follows from a simple and well-known fact: if $\alpha\neq \beta$ (mod $2\pi$), and both $e^{i\alpha}$ and $e^{i\beta}$ are eigenvalues of an operator $B$ of norm one, then every pair $\zeta, \eta$ which are respectively $e^{i\alpha}$- and $e^{i\beta}$-eigenvectors of $B$, are orthogonal. Indeed, if $\langle \zeta,\eta \rangle \neq 0$, let
$\tau=arg(\langle \zeta,\eta \rangle)$. Set $\nu=(\alpha -\beta +\tau)$. WLOG assume $\parallel \zeta\parallel =\parallel \eta \parallel=1$, then
$$\parallel \zeta + e^{i \nu}\eta\parallel ^2=2+2Re (e^{- i\nu}\langle \zeta,\eta \rangle)=2+2|\langle \zeta,\eta \rangle |Re (e^{i(\beta-\alpha)}),$$
while
$$\parallel B(\zeta+e^{i \nu}\eta)\parallel^2=2+2|\langle \zeta, \eta \rangle|>\parallel \zeta + e^{i \nu}\eta\parallel ^2, $$
a contradiction. The proof is finished.

\section{Fourier pairs: proof of Theorems \ref{Theorem2} and \ref{Theorem3}}\label{thm2_and_thm3}

\begin{proof}  
 Statement 1) of Theorem \ref{Theorem2} follows directly from Theorem \ref{Theorem1}. 

To prove 2) let
$$e_0=c_{00}\zeta_0+...+c_{0n-1}\zeta_{n-1}. $$
To make our notation simpler write $\omega=e^{2\pi i/n}$. Then
$$B^me_0=\sum_{j=0}^{n-1}c_{0j}\omega^{mj}\zeta_j, \ m=1,...,n-1,$$
and, therefore,
$$ \langle B^me_0,e_0\rangle = \sum_{j=0}^{n-1} |c_{0j}|^2 \omega^j,$$

Since 
$\langle e_0,e_0\rangle=1, $ we obviously have $\sum_{j=0}^{n-1} |c_{0j}|^2=1$.
Now, it follows from  \eqref{Powers_of_B.1} that $|c_{00}|^2,...|c_{0n-1}|^2$ satisfy the following system of linear equations
$$ \begin{array}{ccccccccc} |c_{00}|^2& + & |c_{01}|^2 & + & ... & + & |c_{0n-1}|^2 & = & 1 \\ 
 |c_{00}|^2&+& \omega |c_{01}|^2 & + &... &+ & \omega ^{n-1}|c_{0n-1}|^2 & = & 0 \\
 . & . &. & . & . & . &. & . & . \\	|c_{00}|^2& + &\omega ^{n-1}|c_{01}|^2 & + & ... & + & \omega ^{(n-1)^2}|c_{0n-1}|^2 & = & 0 .
 \end{array}
$$
Clearly, $|c_{00}|^2=|c_{01}|^2=...=|c_{0n-1}|^2=\frac{1}{n}$ satisfy this system, and, since the determinant of the system is not zero, this is the only solution.

Applying the same argument to $\sigma_p(e^{2\pi mi/n}A,B,e^{2\pi mi/n}AB), \ m=1,2,...n-1$, we obtain that for every $1\leq r\leq n-1$ the coefficients of the decomposition of $e_r$ in the basis $\zeta_0,...,\zeta_{n-1}$ also have the same absolute value $1/\sqrt{n}$:
$$e_r=c_{r0}\zeta_0+...+c_{rn-1}\zeta_{n-1}, \ |c_{r0}|=...=|c_{rn-1}|=\frac{1}{\sqrt{n}} . $$ 

\noindent Since $e_0,...,e_{n-1}$ are orthogonal, the matrix $C$ with columns $\left ( \begin{array}{c} c_{r0} \\ . \\ .\\ . \\ c_{rn-1} \end{array} \right ),$ is unitary, and, therefore, $\sqrt{n}C$ is a complex Hadamard matrix. \end{proof}

We now turn to Theorem \ref{Theorem3}.

\begin{proof}
a) As it was mention earlier, all complex Hadamard matrices of order 3 are similar to the Fourier matrix of order 3. We could use this fact to prove statement a),	 but here is an alternative independent proof that uses our operator relations obtained in the previous section.

Let $e_0,e_1, \ \mbox{and} \ e_2$ be the eigenbasis for $A$, \ $Ae_j=e^{2\pi ji/3} e_j, \ j=0,1,2$.
Relations \eqref{moments.1} and \eqref{moments.2} for $P_0BTBP_0$ and $P_0ABTABP_0$ give
\begin{align}
&\frac{\langle Be_0,e_1\rangle \langle Be_1, e_0\rangle}{e^{2\pi i/3}-1}+\frac{\langle Be_0, e_2\rangle \langle Be_2,e_0\rangle}{e^{4\pi i/3} -1}=0 \\
&\frac{e^{2\pi i/3}}{e^{2\pi i/3}-1}\langle Be_0, e_1\rangle \langle Be_1, e_0\rangle + \frac{e^{4\pi i/3}}{e^{4\pi i /3}-1}\langle Be_0,e_2\rangle \langle Be_2,e_0\rangle =0.
\end{align}
Considering these equations as a system in $\langle Be_0, e_1\rangle \langle Be_1, e_0\rangle $ and  \newline $\langle Be_0,e_2\rangle \langle Be_2,e_0\rangle$ we see that
\begin{align}
\langle Be_0, e_1\rangle \langle Be_1, e_0\rangle=0 \label{system_1} \\
\langle Be_0,e_2\rangle \langle Be_2,e_0\rangle =0.	\label{system_2}
\end{align}

In  a similar way we can get relations analogous to \eqref{system_1}-\eqref{system_2}:
\begin{align}
\langle Be_1, e_0\rangle \langle Be_0,e_1\rangle =0 \label{system_3}\\
\langle Be_1,e_2\rangle \langle Be_2,e_1\rangle =0 \label{system_4}\\
\langle Be_2,e_0\rangle \langle Be_0,e_2\rangle =0 \label{system_5} \\
\langle Be_2, e_1\rangle \langle Be_1,e_2\rangle =0 .	
\end{align}

Equation \eqref{system_1} gives two scenarios:

\noindent 1). $\langle Be_0,e_1\rangle =0$. Then $P_0BP_0=0$ implies
$$Be_0=e^{i\theta_0}e_2. $$
Hence, \eqref{system_2} and $P_2BP_2=0 $ give
\begin{equation}Be_2=e^{i\theta_2}e_1. \label{Be_3}
\end{equation}

Now, $P_1e_1P_1=0$, \eqref{system_4}, and \eqref{Be_3} yield
\begin{equation}
Be_1=e^{i\theta_1}e_0.	
\end{equation}
The condition $B^3=1$ implies $\theta_0+\theta_1+\theta_2=0 \ (mod  \ 2\pi)$ and, therefore,
\begin{align*}&B=\left [ \begin{array}{ccc} 0 & e^{i\theta_1} & 0 \\ 0 & 0 & e^{i\theta_2} \\ e^{i\theta_0} & 0 & 0\end{array} \right ] \\ &= \left [ \begin{array}{ccc} e^{-i\theta_0} & 0 & 0 \\ 0 & e^{i\theta_2} & 0 \\ 0 & 0 & 1 \end{array} \right ] \left [ \begin{array}{ccc} 0 & 1 & 0\\ 0 & 0 & 1 \\ 1 & 0 & 0 \end{array} \right ] \left [ \begin{array}{ccc} e^{i\theta_0} &0 &0 \\ 0 & e^{-i\theta_2} & 0 \\ 0 & 0 & 1 \end{array} \right ].\end{align*}
 Since $A$ commutes with  every matrix which is diagonal in the basis $e_0,e_1,e_2$, the pair $(A,B)$. is unitary equivalent to

$$\left [ \begin{array}{ccc} 1 & 0 & 0 \\ 0 & e^{2\pi i/3} & 0 \\ 0 & 0 & e^{4\pi i /3} \end{array} \right ], \left  [ \begin{array}{ccc}0 & 1 & 0 \\ 0 & 0 & 1 \\ 1 & 0 & 0 \end{array} \right ] . $$
 \vspace{.2cm}
 
 \noindent 2). $\langle Be_1,e_0\rangle =0$.
 
 In this case a similar consideration leads to
 $$Be_0=e^{i\alpha_0}e_1, \ Be_1=e^{i\alpha_1}e_2, \ Be_2=e^{i\alpha_2}e_0,$$
 and, hence,  the pair $(A,B)$ is unitary equivalent to the pair
 $$ \left [ \begin{array}{ccc} 1 & 0 & 0 \\ 0 & e^{2\pi i/3} & 0 \\ 0 & 0 & e^{4\pi i /3} \end{array} \right ] ,  \ \left [ \begin{array}{ccc} 0 & 0 & 1 \\ 1 & 0 & 0\\ 0 & 1 & 0 \end{array} \right ], $$
 which finishes the proof of a).
 \vspace{.2cm}

Before passing to the proof of b)  we prove the following lemma. 

Let $G_n$ be a subgroup of the permutation  group ${\mathcal P}_n$ of all permutations of $(0,1,...,n-1)$, defined by
\begin{align*} 
G_n= \{ P\in {\mathcal P}_n:  
\exists \ 1\leq q\leq n-1 \ \mbox{and} \ 0\leq m\leq n-1  \mbox{ such that $q$ and $n$} \\ \mbox{are mutually prime and} \ \forall \ 0\leq j\leq n-1, \  
P(j)=(qj+m ) \ (mod \ n)\}. 
\end{align*}

\begin{lemma}\label{lemma}
Let  $P_1$ and $P_2$ be two permutation matrices of order $n$. Write
$$ B(P_1,P_2)=\frac{1}{n}P_2^*F_n^*P_1^*AP_1F_nP_2,$$
where, as before, $A$ is the diagonal matrix with roots of unity in the increasing order of arguments on the diagonal. 	If either $P_1\in G_n$ or $P_2\in G_n$, then 
$$\sigma_p(A,B(P_1,P_2),AB(P_1,P_2))=\{x^n+y^n+(-1)^{n-1}z^n=1\} $$
and
\begin{itemize}
\item[a)] If $P_1\in G_n$ the there exists a permutation $P\in {\mathcal P}_n$ such that the pair $(A,B(P_1,P_2))$ is unitary equivalent	to $(P^*AP,\widehat{B}_n)$ ( matrix $\widehat{B}_n$ was defined by \eqref{B_hats}).
\item[b)] If $P_2\in G_n$, there exists a permutation $P\in {\mathcal P}_n$ such that the pair $(A,B(P_1,P_2))$ is unitary equivalent to $(\widehat{B}_n,P^*AP)$.
\end{itemize}

\end{lemma}

\begin{proof}
a). Suppose that $P_1\in G_n$, \ $P_1(j)=qj+m, \ j=0,...,n-1$, and $q$ and $n$ are mutually prime. If $B(P_1,P_2)=[b_{kl}]_{k,l=0}^{n-1}$, then (here again $\omega=e^{2\pi i/n}$ is the prime $n$-th root of unity)
\begin{align}b_{kl}=&\frac{1}{n}\sum_{j=0}^{n-1} \omega^{P_1(j)[P_2(l)-P_2(k)]+j} = \frac{1}{n}\sum_{j=0}^{n-1} \omega^{(qj+m)[P_2(l)-P_2(k)]+j}. \nonumber \\ &=\frac{1}{n}\omega^{m[P_2(l)-P_2(k)]}\sum_{j=0}^{n-1}\omega^{j[\tilde{P}_2(l)-\tilde{P}_2(k)+1] } \label{roots}
\end{align}
where $\tilde{P}_2(s)=qP_2(s) \ (mod \ n)$. Since $q$ and $n$ are mutually prime, $\tilde{P}_2$ is a permutation of $0,1,...,n-1$. Furthermore, unless $\tilde{P}_2(l)-\tilde{P}_2(k)+1=0 \ (mod \ n)$, the last sum in \eqref{roots} is equal to zero. Indeed, if $(\tilde{P}_2(l)-\tilde{P}_2(k)+1)$ is mutually prime with n, then $\omega^{j[\tilde{P}_2(l)-\tilde{P}_2(k)+1] }$ runs over all $n$-th roots of unity as $j$ runs from $0$ to $n-1$. If $(\tilde{P}_2(l)-\tilde{P}_2(k)+1)$ and $n$ have a non-trivial common divisor, let $s$ be their greatest common divisor, so that $(\tilde{P}_2(l)-\tilde{P}_2(k)+1)=sr, \ n=st,  $  with $r$ and $t$ being mutually prime. In this case as $j$ runs from $0$ to $n-1$, \  $\omega^{j[\tilde{P}_2(l)-\tilde{P}_2(k)+1]} $ runs $s$ times over the set of $t$-th roots of unity. In both cases the sum is zero. 

Thus, if $k$ is fixed, there is only one $l=l(k)$ that satisfies
$$\tilde{P}_2(k)=\tilde{P}_2(l(k))+1 \ (mod \ n).\, $$
which means that each row of $B(P_1,P_2)$ has only one non-zero entry, and this entry is equal to 
$$\omega^{m[P_2(l(k))-P_2(k)]}=\omega^{mq^{-1}[\tilde{P}_2(l(k))-\tilde{P}_2(k)]}=\omega^{-mq^{-1}},$$
where $q^{-1}$ is taken in the sense ${\mathbb Z}_n$ (since $q$ and $n$ are mutually prime, $q$ is invertible in ${\mathbb Z}_n$). Moreover, this  non-zero entry is the same $n$-th root of unity for each row. Fixing $l$ we obtain that each column of $B(P_1,P_2)$ also contains only one non-trivial element which is equal to the same $n$-th root of unity. 

Of course, it means that
$$ B(P_1,P_2)= \omega^{-mq^{-1}}I C(\tilde{P}_2)=\Lambda^*C(\tilde{P}_2)\Lambda,$$
where $C(\tilde{P}_2)$ is the matrix whose all non-trivial entries are ones, and an entry $c_{kl}$ is non-trivial if and only if
$$\tilde{P}_2(l)-\tilde{P}_2(k)=-1 \ (mod \ n), $$
 and $\Lambda$ is a diagonal matrix whose each diagonal entree is an $n$-th root of unity (it is very easy to prove that such $\Lambda $ exists). Of course,
$$C(\tilde{P}_2)= \tilde{P}_2^*\widehat{B}_n\tilde{P}_2 ,$$
and here we denoted by the same symbol $\tilde{P}_2$ the corresponding permutation matrix.   As a result we obtain 
 \begin{equation}\label{equivalence1}B(P_1,P_2)= \Lambda^*\tilde{P}_2^*\widehat{B}_n\tilde{P}_2\Lambda,\end{equation}
 and, so
 \begin{align*}
 &xA+yB(P_1,P_2)+zAB(P_1,P_2)-I \\
 &=xA+y\Lambda^*\tilde{P}_2^*\widehat{B}_n\tilde{P}_2\Lambda +zA\Lambda^*\tilde{P}_2^*\widehat{B}_n\tilde{P}_2\Lambda -I \\
 &=\Lambda^*\tilde{P}_2^* \left [ x\tilde{P}_2\Lambda A\Lambda^*\tilde{P}_2^*+y\widehat{B}+z\tilde{P}_2\Lambda A\Lambda^*\tilde{P}_2^*\widehat{B}_n-I \right ]\tilde{P}_2\Lambda.
 \end{align*} 
 Since both $A$ and $\Lambda$  are diagonal, they commute, and we obtain
 \begin{align}
 &xA+yB(P_1,P_2)+zAB(P_1,P_2)-I \nonumber \\ &=\Lambda^*\tilde{P}_2^* \left [ x\tilde{P}_2 A\tilde{P}_2^*+y\widehat{B}_n+z\tilde{P}_2 A\tilde{P}_2^*\widehat{B}_n-I \right ]\tilde{P}_2\Lambda. \label{eq.2}
 \end{align}
 Therefore, 
 \begin{equation}\label{eq.1}
 \sigma_p(A,B(P_1,P_2),AB(P_1,P_2))=\sigma_p( \tilde{P}_2 A\tilde{P}^*_2, \widehat{ B}_n,\tilde{P}_2 A\tilde{P}_2^*\widehat{B}_n).
 \end{equation}
 The matrix $\tilde{P}_2 A\tilde{P}_2^*$ is a diagonal matrix with $n$-th roots of unity on the diagonal permuted according to $\tilde{P}_2$, so one can directly check that
 $$\sigma_p( \tilde{P}_2 A\tilde{P}_2^*, \widehat{ B}_n,\tilde{P}_2 A\tilde{P}_2^*\widehat{B}_n)=\{x^n+y^n+(-1)^{n-1}z^n=1\}.$$ 
 It also follows from \eqref{equivalence1} that the pair $(A,B(P_1,P_2))$ is unitary equivalent to $(\tilde{P}_2 A\tilde{P}_2^*, \widehat{ B}_n)$.

b). Now suppose that $P_2\in G_n$, \ $P_2(j)=qj+m$, where $q$ and $n$ are mutually prime. In this case
 \begin{equation}\label{reduced_to_one}
 b_{kl}=\frac{1}{n}\sum_{j=0}^{n-1} \omega^{P_1(j)q(l-k)+j}=\frac{1}{n}\sum_{j=1}^{n-1} \omega^{\tilde{P}_1(j)(l-k)+j}, 
 \end{equation}
 where $\tilde{P}_1(j)=qP_1(j) \ (mod \ n)$. Again, since $q$ and $n$ are mutually prime, $\tilde{P}_1$ is a permutation in ${\mathcal P}_n$, and, hence, \eqref{reduced_to_one} shows that
 $$B(P_1,P_2)= B(\tilde{P}_1,I),$$
 where $I$ is the identity permutation. Thus,
$$B(P_1,P_2)=\frac{1}{n}F_n^*\tilde{P}_1^*A \tilde{P}_1F_n, $$ 
and 
\begin{align*}
&xA+yB(P_1,P_2)+zAB(P_1,P_2)-I \\
&=\frac{1}{\sqrt{n}}F_n^*\left (xF_nAF_n^*+y\tilde{P}_1^*A\tilde{P}_1+z(F_nAF_n^*)(\tilde{P}_1^*A\title{P}_1)-I\right )\frac{1}{\sqrt{n}}F_n \\
&=\frac{1}{\sqrt{n}}F_n^*\left ( x\widehat{\widehat{B}}+y\tilde{P}_1^*A\tilde{P}_1+z\widehat{\widehat{B}}\tilde{P}_1^*A\tilde{P}_1-I \right )\frac{1}{\sqrt{n}} F_n.
\end{align*}
This, of course, implies,
$$\sigma_p(A,B(P_1,P_2),AB(P_1,P_2))=\sigma_p(\widehat{\widehat{B}}, (\tilde{P}_1^*A\tilde{P}_1),\widehat{\widehat{B}}(\tilde{P}_1^*A\tilde{P}_1)), $$ 
and that the pair $(A,B)$ is unitary equivalent to $(\widehat{\widehat{B}},(\tilde{P}_1^*A\tilde{P}_1))$.  Now  statement b) of Lemma \ref{lemma} follows from \eqref{hats}.

Finally, again, $\tilde{P}_1^*A\tilde{P}_1$ is a diagonal matrix whose entries on the main diagonal are $n$-th roots of unity permuted according to $\tilde{P}_1$. It is now very easy to check that 
$$\sigma_p(\widehat{\widehat{B}},(\tilde{P}_1^*A\tilde{P}_1),\widehat{\widehat{B}}(\tilde{P}_1^*A\tilde{P}_1))=\{x^n+y^n+(-1)^{n-1}z^n=1\}. $$
Lemma \ref{lemma} is completely proved.
\end{proof}

Now we are able to finish the proof of Theorem \ref{Theorem3}, section b). 

Let $h$ be the complex Hadamard transition matrix  from Theorem \ref{Theorem2}, so that
$$B=\frac{1}{n}h^*Ah. $$ 
First we observe that if $\Lambda_1,\Lambda_2$ are diagonal matrices with unimodular entries on the main diagonal, then
$$\tilde{h}= \Lambda_1h\Lambda_2, \ \mbox{and} \ \tilde{B}=\frac{1}{n}\tilde{h}^*A\tilde{h}=\frac{1}{n}\Lambda_2^*h^*\Lambda_1^*A\Lambda_1h\Lambda_2=\frac{1}{n}\Lambda_2^*B\Lambda_2, $$
($A$ and $\Lambda$ commute since both of them are diagonal), so that
$$(A,\tilde{B})=(A,\Lambda_2^*B\Lambda_2)=\Lambda_2^*(A,B)\Lambda_2 ,$$
and we see that the pairs $(A,\tilde{B})$ and $(A,B)$ are unitary equivalent.  For this reason  in our consideration we may omit the diagonal matrices $\Lambda_1,\Lambda_2$ 
 in \eqref{hadamard}.
 
 Let $n=4 \ \mbox{or} \ 5$. If $h=P_1F_nP_2 $, where at least one of $P_1,P_2$ is in $G_n$ with corresponding $n$, then by Lemma \ref{lemma},  $\sigma_p(A,B,AB)=\{x^n+y^n+(-1)^{n-1}z^n=1\}$, and the pair $(A,B)$ is unitary equivalent to either $(P^*AP,\widehat{B})$ or $(\widehat{B}, P^*AP)$
with some $P\in {\mathcal P}_n$, so for these matrices the statement is established. We used SAGEMATH software to verify that in both cases $n=4,5$ for any other $h$ the joint spectrum 
 $\sigma_p(A,\frac{1}{n}h^*Ah,\frac{1}{n}Ah^*Ah)$ is different from the Fourier surface $\{x^n+y^n+(-1)^{n-1}z^n=1\}$. The corresponding simple algorithms are in section \ref{last}.  The first two algorithms show that for the case $n=4$, a Hadamard matrix must be similar to the Fourier matrix $F_4$, and also it must be in the form given by Lemma \ref{lemma},  otherwise the coeffiecent of $z^2$ is non-trivial.  The third algorithm verifies that if $n=5$, for all pairs of permutations, $P_1,P_2$, the  joint spectrum of $A,B,AB$ coincides with the Fourier surface \eqref{Fourier1}, $\sigma_p(A,B,AB) = \{x^5+y^5+z^5 = 1\}$, only when $P_1 \in G_5$ or $P_2 \in G_5$. 
 The proof of Theorem \ref{Theorem3} is finished.

\end{proof}

\section{Rigidity theorem for Fourier surfaces: proof of Theorem \ref{Theorem4}}\label{thm4}

\vspace{.2cm} 

\begin{proposition}\label{A_invariant}
 If $A$ and $B$ satisfy the conditions of Theorem \ref{Theorem4}, then
\begin{align}
&P_0B^rABP_0=0, \ r=1,...,n-2	\label{A_diagonalized.1} \\
&P_0B^{n-1}ABP_0=e^{2\pi i/n}P_0 \label{A_diagonalized.2}
\end{align}
	\end{proposition}

\begin{proof}
The proof resembles those of Propositions \ref{first-fact} and \ref{second-fact}. 
We prove \eqref{A_diagonalized.1}  by induction in $r$. First, suppose that $r=1$. Since there are no monomials of positive degree less than $n$ in $$\{ x^n+y^n+(-1)^{n-1}(e^{2\pi i/n}z_1+z_2)^n=1\},$$
relation \eqref{residue_in_coefficients} applied for the residue of the term corresponding to $z_1z_2$ gives
$$P_0ABTBAP_0+P_0BATABP_0=0.$$
Since $P_0A=AP_0=P_0$, this relation can be written as
$$ P_0BTBP_0+P_0BATABP_0=0.$$
Now, \eqref{TAB} yields
$$P_0BTBP_0+P_0BA(B-P_0B+TB)P_0=0. $$
The first term $P_0BTBP_0$ vanishes by \eqref{moments.1}, so that
$$0=P_0BABP_0-P_0BA P_0BP_0+P_0BATBP_0.$$
The second term in the last equality vanishes by \eqref{moments.1} ($P_0BP_0=0$), and the last tem can be written
$$P_0BATBP_0=P_0(B-BP_0+BT)BP_0=P_0B^2P_0-P_0BP_0BP_0+P_0BTBP_0 .$$
$P_0B^2P_0=0$ by \eqref{Powers_of_B.1}, and $P_0BP_0BP_0=P_0BTBP_0=0$ again by \eqref{moments.1} Thus,
\begin{equation}\label{r=1}  P_0BABP_0=0,\end{equation}
and the result is established for $r=1$.

Suppose that the result holds for all $1\leq r\leq l \ (l<n-2)$, that is $P_0B^rABP_0=0$ for all $r\leq l$. Let us prove that $P_0B^{l+1}ABP_0=0$.

First, we claim that for $1\leq r\leq l$ the induction assumption implies the following relation 

\begin{equation}\label{through_BAT-BT}
0=P_0B^rABP_0=P_0(BAT-BT)^rABP_0. 
\end{equation}
Indeed, it follows from \eqref{A_and_T} that
$$I=AT+P_0-T, $$
and, therefore, for every $r$
\begin{align}
 &P_0B^rABP_0 \nonumber \\
 &=P_0\underbrace{B(AT+P_0-T)B(AT+P_0-T)...B(AT+P_0-T)}_{r}ABP_0. 
 \end{align}
If $r=1$, then we have by \eqref{r=1}
\begin{align*}
&0=P_0BABP_0=P_0(BAT+BP_0-BT)ABP_0 \\
&=P_0(BAT-BT)ABP_0+P_0BP_0ABP_0=P_0(BAT-BT)ABP_0. 
\end{align*}
For $r=2$ using $AP_0=P_0$ and \eqref{moments.1} we have 
\begin{align*}
&0=P_0B^2ABP_0=P_0(BAT+BP_0-BT)(BAT+BP_0-BT)ABP_0	 \\
&=P_0(BAT+BP_0-BT)(BAT-BT)ABP_0 \\ 
&+P_0(BAT+BP_0-BT)BP_0ABP_0 \\
&=P_0(BAT+BP_0-BT)(BAT-BT)ABP_0 \\
&+P_0(BAT+BP_0-BT)BP_0BP_0\\
&=P_0(BAT+BP_0-BT)(BAT-BT)ABP_0  \\ 
&=P_0(BAT-BT)^2ABP_0+P_0BP_0(BAT-BT)ABP_0 \\ 
&=P_0(BAT-BT)^2ABP_0+P_0B(P_0BABP_0)=P_0(BAT-BT)^2ABP_0.
\end{align*}
Similarly, we show
\begin{align*}P_0B^3ABP_0=P_0(BAT-BT)^3ABP_0+P_0B(P_0(BAT-BT)^2ABP_0 \\= P_0(BAT-BT)^3ABP_0+P_0B(P_0B^2ABP_0)  
=P_0(BAT-BT)^3ABP_0
\end{align*}

We continue this way and obtain \eqref{through_BAT-BT}. The above proof of the equality \eqref{through_BAT-BT} shows that
\begin{align}
&P_0B^{l+1}ABP_0=P_0(BAT+BP_0-BT)(BAT-BT)^lABP_0 \nonumber\\
&=P_0(BAT-BT)^{l+1}ABP_0+(P_0BP_0)(BAT-BT)^lABP_0 \nonumber \\
&=P_0(BAT-BT)^{l+1}ABP_0. \label{through_BAT-BT.1}
\end{align}

We now introduce the following operators. 
%For $0\leq r\leq l+1$ set
\begin{align*}
{\mathcal E}_1(r,s)=\left (\sum E_1...E_{r+s}\right )AB,	\  \mbox{where} \ E_j=\left \{ \begin{array}{c} BAT \\ BT \end{array}\right. ,
\end{align*}
where the sum is taken over all products $E_1...E_{r+s}$ which contain $r$ \ ($BAT$) terms,  $s$ - $(BT)$ terms, and $r+s\leq n-1$;
$${\mathcal E}_2(r,s)=\left (\sum E_1...E_{r+s}\right )B, \ \mbox{where} \ E_j=\left \{ \begin{array}{l} BAT \\BT\\ ABT \end{array} \right. ,$$
 and here the sum is taken over all  $E_1...E_{r+s}$ 	where $ABT$ occurs once and $BAT$ - $(r)$ times, a $BT$ - $(s)$ times, and $r+s\leq n-2$
 % finally write
%% $${\mathcal E}_3(r)=\sum E_1...E_{l+1}BA, \ E_j=\left \{ \begin{array}{l} BAT \\BT\\ ABT \end{array} \right. ,$$
 %and in each product of this sum $(ABT)$ occurs once in $E_1...E_{l+1}$, \ $(BT)$ - $r$ times, and $(BAT) \  - \ (l-r)$ times.
%Of course, 
%\begin{equation}\label{zero_D}
%{\mathcal E}_2(0)=0={\mathcal E}_3(l+1)
%\end{equation}.

The relations between these operators are derived from \eqref{residue1} and \eqref{residue_in_coefficients} the following way.

The polynomial $x^n+y^n+(-1)^{n-1}(e^{2\i i/n} z_1+z_2)^n-1$ does not contain any monomials $y^{m_1}z_1^{m_2}z_2^{m_3}$ with $0<m_1+m_2+m_3<n$. Therefore, in this case, if $r+s\leq n-2$, \eqref{residue1}, \eqref{residue_in_coefficients}, and  $AP_0=P_0$ applied to 
 $\psi_{s,1,r}$ and written in terms of operators ${\mathcal E}_j(r,s)$ yield
\begin{align}
&P_0\left ({\mathcal E}_1(r,s)+{\mathcal E}_2(r,s-1)+{\mathcal E}_2(r-1,s)\right )P_0=0 \ \mbox{if} \ r\neq 0,s\neq 0	\label{r,s_term}\\
&P_0 ( {\mathcal E}_1 (0,s)+{\mathcal E}_2(0,s-1))P_0=0 \label{0,s_term}\\
&P_0 ({\mathcal E}_1(r,0)+{\mathcal E}_2(r-1,0))P_0=0. \label{r,0_term}
\end{align}

The coefficient for $z_1z_2^{n-1}$ in the same polynomial is $(-1)^{n-1}ne^{2\pi i/n}$, so that \eqref{residue1} and \eqref{residue_in_coefficients} imply
\begin{equation}\label{n-1,0_term}
P_0({\mathcal E}_1(n-1,0)+{\mathcal E}_2(n-2,0))P_0=e^{2\pi i/n}P_0.	
\end{equation}

Equations \eqref{r,s_term} - \eqref{n-1,0_term} give

\begin{align}
&P_0{\mathcal E}_1(r,s)P_0=-P_0(	{\mathcal E}_2(r,s-1)+{\mathcal E}_2(r-1,s) )P_0, \ r\neq 0, s\neq 0 \label{r,s_term.1}\\
&P_0{\mathcal E}_1(0,s)P_0=-P_0{\mathcal E}_2(0,s-1)P_0 \label{0,s_term.1}\\
&P_0{\mathcal E}_1(r,0)P_0=-P_0{\mathcal E}_2(r-1,0)P_0 \label{r,0_term.1}\\
&P_0{\mathcal E}_1(n-1,0)P_0=e^{2\pi i/n}P_0-P_0{\mathcal E}_2(n-2,0)P_0.\label{n-1,0_term.1}
\end{align}

Further, \eqref{through_BAT-BT.1} can be written in the form
\begin{equation}\label{telescope} P_0B^{l+1}ABP_0=\sum_{s=0}^{l+1} (-1)^s P_0{\mathcal E}_1(l+1-s,s)P_0.\end{equation}
Now, for $l+1\leq n-2$ relations \eqref{r,s_term.1} - \eqref{r,0_term.1} show that the sum in the right hand side of \eqref{telescope} telescopes to zero, which finishes the proof of \eqref{A_diagonalized.1}.

\begin{comment}
It easily follows from \eqref{through_BAT-BT.1} that
\begin{equation}\label{decomposition_through_D}
P_0B^{l+1}ABP_0=P_0{\mathcal E}_1(0)P_0-P_0{\mathcal E}_1(1)P_0+_...+(-1)^{l+1}P_0{\mathcal E}(l+1)P_0.	
\end{equation}

We also remark that since $AP_0=P_0$,
\begin{equation}\label{r_to_r+1}
P_0{\mathcal E}_3(r)P_0=P_0{\mathcal E}_2(r+1)P_0. 
\end{equation}

Furthermore, relation \eqref{residue_in_coefficients} for $m_2=r, \ m_3=1$, and $m_4=l+1-r$,  applied to the  polynomial   $x^n+y^n+(-1)^{n-1}(e^{2\pi i/n}z_1+z_2)^n-1$   implies
\begin{equation}\label{relation_between_D}
P_0{\mathcal E}_1(r)P_0+P_0{\mathcal E}_2(r)P_0+P_0{\mathcal E}_3(r)P_0=0, 
\end{equation}
so that
$$P_0{\mathcal E}_1(r)P_0=-P_0{\mathcal E}_2(r)P_0-P_0{\mathcal E}_3(r)P_0. $$
Hence, using \eqref{zero_D}, relation \eqref{decomposition_through_D} can be written as
\begin{align*}
&P_0B^{l+1}ABP_0= -P_0{\mathcal E}_3(0)P_0+\sum_{r=1}^l (-1)^{r-1}\left ( P_0{\mathcal E}_2(r)P_0+P_0{\mathcal E}_3(r)P_0 \right ) \\
& + (-1)^l P_0{\mathcal E}_2(l+1)P_0.
\end{align*} 
Now, \eqref{r_to_r+1} implies that the last sum telescopes to zero. 

\end{comment}

The proof of \eqref{A_diagonalized.2} is very similar to the one of \eqref{A_diagonalized.1}. First, using \eqref{A_diagonalized.1} we use the same proof to show that
$$P_0B^{n-1}ABP_0=P_0(BAT-BT)^{n-1}ABP_0.$$
Also, since there are no monomials $y^sz_1z_2^r$ with $s>0$, relations \eqref{r,s_term} - \eqref{r,0_term} are valid for $s\geq 1$. We use these relations along with \eqref{n-1,0_term.1} in a similar way and obtain \eqref{A_diagonalized.2}. We are done.

\begin{comment}
We define operators ${\mathcal E}_1(n-1,r), \ {\mathcal E}_2(n-1.r)$ and ${\mathcal E}_3(n-1,r)$ by exactly the same formulas as above when $l+1=n-1$, and, of course,  relations  \eqref{zero_D}, \eqref{decomposition_through_D}, \eqref{r_to_r+1} hold for $l+1=n-1$. The only difference here is that, since $\{x^n+y^n+(-1)^{n-1}(e^{2\pi i/n}z_1+z_2)^n=1\}$ contains the monomial $ne^{2\pi i/n}z_1z_2^{n-1}$, so that an application of \eqref{residue_in_coefficients} yields that in this case relation \eqref{relation_between_D} changes to
\begin{align*}
&P_0{\mathcal E}_1(n-1,0)P_0+P_0{\mathcal E}_3(n-1,0)P_0=P_0 \\
&P_0{\mathcal E}_1(,n-1,r)P_0+P_0{\mathcal E}_2(n-1,r)P_0+P_0{\mathcal E}_3(n-1,r)P_0=0, \ 1\leq r\leq n-1,	
\end{align*}
and, therefore,
\begin{align*}
&P_0{\mathcal E}_1(n-1,0)P_0=P_0-P_0{\mathcal E}_3(n-1,0)\\
&P_0{\mathcal E}_1(n-1,r)P_0=-P_0{\mathcal E}_2(n-1,r)P_0-P_0{\mathcal E}_3(n-1,r)P_0, \ 1\leq r\leq n-1.	
\end{align*}
\end{comment}

\end{proof}

\vspace{.2cm}

\textbf{Remark} Of course, since $\parallel B\parallel =1$, \eqref{A_diagonalized.1} follows from \eqref{A_diagonalized.2}, but, as we saw above, that would not simplify the proof.

\vspace{.2cm}

We are now ready to prove Theorem \ref{Theorem4}.

\begin{proof}

Let $e_0,...,e_{n-1}$ be an orthonormal eigenbasis for the restriction of $A$ to $L$, the subspace of Theorem \ref{Theorem1}. By Theorem \ref{Theorem1} the restriction of $B$ to $L$ is unitary. Hence, for $m\neq l, \ m,l\leq n-1$
\begin{align*}
&\langle B^me_0, B^le_0\rangle=\langle B^{n-l}B^me_0,e_0\rangle=\langle B^{n-l+m}e_0,e_0\rangle  \\
&= \langle P_0B^{n-l+m}P_0e_0,e_0\rangle=0, 
\end{align*}
so that $e_0,Be_0,...,B^{n-1}e_0$ form an orthonormal basis of $L$.

\begin{comment}
Pick up an orthonormal basis $e_1,...,e_k$ of the eigenspace of $A$ with eigenvalue $1$. Since $B$ is unitary on $L$, by Proposition \ref{eigenvector} \newline $e_1,Be_1,...,B^{n-1}e_1,e_2,Be_2,...,B^{n-1}e_2,...,e_k,Be_k,...,B^{n-1}e_k$ 	form an orthonormal basis of $L$.
Indeed, if $i\neq j$, then

because for $l\neq m \ P_0B^{n-l+m}P_0=0$, and for $l=m$ we have $P_0B^nP_0=P_0$, so that 
$$\langle B^m e_i,B^me_j \rangle =\langle P_0B^nP_0e_i,e_j\rangle=\langle e_i,e_j\rangle =0. $$
Furthermore, if $l\neq m$,
$$\langle B^me_i,B^le_i\rangle=\langle B^{n-l+m}e_i,e_i\rangle=0$$
by Corollary \ref{orthogonal.1}.

Now, every subspace $M_i=span\{e_i, Be_i,...,B^{n-1}e_i\}, \ i=1,...,k$ is obviously $B$-invariant and the matrix of $B$ in this basis is
$$\left [ \begin{array}{ccccc} 0 & 0 & ... & 0 & 1 \\ 1 & 0 & ... & 0 & 0 \\ 0 & 1 & ... & 0 & 0\\. & . &. & . & . \\0 & 0 & ... & 1 & 0 \end{array} \right ] . $$
\end{comment}

Proposition \ref{A_invariant} implies that each $B^me_0$ is an eigenvector for $A$. To see this we remark that  Proposition \ref{A_invariant} yields
$$ \langle ABe_0, B^me_0\rangle =\langle P_0B^{n-m}ABP_0e_0,e_0\rangle =\left \{ \begin{array}{cc} 0 & m\neq 1 \\
  e^{2\pi i/n}	& m=1 . \end{array} \right.
$$
This shows that $ABe_0=e^{2\pi i/n}Be_0$, so that $Be_0$ is a $e^{2\pi i/n}$-eigenvector for the operator $A$, and, therefore, $Be_0$ is a 1-eigenvector of $(e^{2\pi i (n-1)/m}A)$. 
Since the joint spectrum \newline 
$$\sigma_p(e^{2\pi (n-1) i/n}A,B,e^{2\pi (n-1) i/n}AB,B(e^{2\pi (n-1) i/n}A))$$ 
is the same as  $\sigma_p(A,B,AB,BA)$, the above argument is applicable and shows that $B^2e_0=B(Be_0)$ is an eigenvector of $(e^{2\pi (n-1) i/n}A)$ with eigenvalue $e^{2\pi i/n}$. 
Therefore,
$$e^{2\pi i (n-1)/n}AB^2e_0=((e^{2\pi (n-1) i/n}A)B)(Be_0)=e^{2\pi i/n}B^2e_i, $$ 
so that
$$AB^2e_0=e^{4\pi i/n}B^2e_0. $$
We proceed inductively this way and show that 
$$AB^me_i=e^{2\pi m i/n} B^me_i, \ m=1,...,n-1,$$
so the restriction of the pair $(A,B)$ to $L$ in the basis $e_0,Be_0,...,B^{n-1}e_0$ is exactly what was declared in the statement of Theorem \ref{Theorem4}. The verification of the fact that the transition matrix is $\frac{1}{\sqrt{n}}\Lambda_1F_n\Lambda_2$ for some diagonal matrices $\Lambda_{1,2}$ with the entrees on the main diagonal having absolute value one, is straightforward. 
%We just remark that passing from the bases  $e_0,...,e_{n-1}$ and $\zeta_0,...,\zeta_{n-1}$ to $e^{i\theta_0}e_0,...,e^{i\theta_{n-1}}e_{n-1}$ and $e^{i\varphi_0}\zeta_0,...,e^{i\varphi_{n-1}}\zeta_{n-1}$ respectively with properly chosen arguments $\theta_j$ and $\varphi_j$, makes it possible to get rid of the matrices $\Lambda_{1,2}$ and have $\frac{1}{\sqrt{n}}F_n$ as the transition matrix.
\end{proof}

\vspace{.2cm}

Corollary \ref{cor} follows directly from Theorem \ref{Theorem4}.

\section{Algorithms Supporting the Proof of Theorem 1.10}\label{last}

\subsection{$4\times 4$ Hadamard Matrices}\

Here is the program constructed to determine the coefficents of $z^2$ using all permutations of $S4$, and shows that $H$ must be a Fourier Matrix.

\begin{verbatim}
S4 = SymmetricGroup(4)
var('z')
t = var('t', domain=RR)
H = matrix(SR, 4, 4, [1,1,1,1,
                    1,cos(t)+i*sin(t),-1,-1*(cos(t)+i*sin(t)),
                    1,-1,1,-1,
                    1,-1*(cos(t)+i*sin(t)),-1,cos(t)+i*sin(t)])
Hstar = matrix(SR, 4, 4, [1,1,1,1,
                    1,cos(t)-i*sin(t),-1,-1*(cos(t)-i*sin(t)),
                    1,-1,1,-1,
                    1,-1*(cos(t)-i*sin(t)),-1,cos(t)-i*sin(t)])
A = matrix(SR,4,4,[1,0,0,0,
                 0,i,0,0,
                 0,0,i^2,0,
                 0,0,0,i^3])
C = []
for l in S4:
    for p in S4:
        PermH = p.matrix()*H*l.matrix()
        PermHStar = l.matrix().transpose()*Hstar*p.matrix().transpose()
        B = 1/4*PermHStar*A*PermH
        Pencil = z*A*B-I
        pencil = Pencil.determinant()
        pencil.collect(z)
        c = pencil.coefficient(z,2).real()
        if c not in C:
            C.append(c)
            

print(C)
\end{verbatim}

The following algorithm shows that nothing outside of the subgroup, $G_4$, gives us the algebraic hypersurface $\{x^4+y^4-z^4=1 \}$.

\begin{verbatim}
S4 = SymmetricGroup(4)
g1 = S4("(1,2,3,4)")
g2 = S4("(1,3)")
G4 = S4.subgroup([g1,g2])
var('x y z1')
A = matrix(SR,4,4,[1,0,0,0,
                 0,i,0,0,
                 0,0,i^2,0,
                 0,0,0,i^3])
F = matrix(SR, 4, 4, lambda j,k: i^(j*k))
Fstar = F.conjugate_transpose()
JointSpectrum = -x^4 - y^4 + z1^4 + 1
L = []
for l in S4:
    for p in S4:
        if l not in G4 and p not in G4:
            PermF = p.matrix()*F*l.matrix()
            PermFStar = l.matrix().transpose()*Fstar*p.matrix().transpose()
            B = 1/4*PermFStar*A*PermF
            Pencil = x*A+y*B+z1*A*B-I
            pencil = Pencil.determinant()
            if pencil == JointSpectrum:
                L.append([l,p])

print(L)
\end{verbatim}

\subsection{$5\times 5$ Fourier Matrices}
Here the following algorithm shows that this subgroup, $G_5$, has the only permutations that produce this joint spectrum. In order to make the coeffiecients easier to compute, we used the universal cyclotomic field.  

\begin{verbatim}
UCF = UniversalCyclotomicField()
g = UCF.gen(5)
S5 = SymmetricGroup(5)
h1 = S5("(1,2,3,4,5)")
h2 = S5("(1,2,4,3)")
G5 = S5.subgroup([h1,h2])
R = PolynomialRing(UCF,3,'x')
A = matrix(R, 5, 5, [1,0,0,0,0,
                      0,g,0,0,0,
                      0,0,g^2,0,0,
                      0,0,0,g^3,0,
                      0,0,0,0,g^4])
I = matrix(R, 5, 5, [1,0,0,0,0,
                      0,1,0,0,0,
                      0,0,1,0,0,
                      0,0,0,1,0,
                      0,0,0,0,1])
F = matrix(R,5,5,lambda i,j: g^(i*j))
Fstar = matrix(R,5,5, lambda i,j: g^(5-i*j))
x0,x1,x2 = R.gens()
JointSpectrum = x0^5 + x1^5 + x2^5 - 1
L = []
for p in S5:
    for l in S5:
        if p not in G5 and l not in G5:
            PermF = p.matrix()*F*l.matrix()
            PermFStar = l.matrix().transpose()*Fstar*p.matrix().transpose()
            B = 1/5*PermFStar*A*PermF
            Pencil = x0*A+x1*B+x2*A*B-I
            pencil = Pencil.determinant()
            if pencil == JointSpectrum:
                L.append([l,p])
                
print(L)
\end{verbatim}

\end{document}